\documentclass[12pt,reqno]{amsart}
\usepackage{blindtext}
\usepackage{comment}

\usepackage[OT2,T1]{fontenc}
\newcommand\textcyr[1]{{\fontencoding{OT2}\selectfont #1}}

\makeatletter

\usepackage[a4paper]{geometry}
\usepackage[dvipsnames]{xcolor}
\usepackage{mathtools,amssymb,bm,amsthm} 
\usepackage[dvipsnames]{xcolor} 
\usepackage{graphicx}
\usepackage{comment}
\usepackage{tikz-cd}
\usepackage[e]{esvect} 
\usepackage{breakcites} 
\usepackage{hyperref}
\usepackage{calrsfs}
\usepackage{hyperref} 
\hypersetup{
    colorlinks = true,
    linkcolor = {RedOrange},
    citecolor = {Green},
} 
\usepackage{caption} 
\usepackage{subcaption} 
\usepackage[cal=euler]{mathalfa}

\numberwithin{equation}{part}

\usepackage{titlesec}
\titleformat{\part}
{\normalfont\normalsize\bfseries\centering}{\thepart.}{1em}{} 

\titleformat{\section}
{\normalfont\normalsize\bfseries\centering}{\thepart.\thesection}{1em}{} 

\usepackage{chngcntr}
\counterwithin*{section}{part}

\usepackage{overpic}

\theoremstyle{plain}

\newtheorem{theorem}{Theorem}[part]
\newtheorem{corollary}[theorem]{Corollary}
\newtheorem{lemma}[theorem]{Lemma}
\newtheorem{proposition}[theorem]{Proposition}
\theoremstyle{definition}
\newtheorem{definition}[theorem]{Definition}

\newtheorem{remark}[theorem]{Remark}

\DeclarePairedDelimiter\ceil{\lceil}{\rceil}
\DeclarePairedDelimiter\floor{\lfloor}{\rfloor}

\newcommand{\ZZ}{\mathbb{Z}}
\newcommand{\bO}{\mathrm{O}}
\newcommand{\cC}{\mathcal{C}}
\newcommand{\Pp}{\mathcal{P}}

\title{Single self-insterction words on the once-punctured torus and their counting}

\author{David Fisac}
\address{David Fisac, DMATH, University of Luxembourg, Esch-sur-Alzette, Luxembourg \&  Universitat Aut\`onoma de Barcelona, Barcelona, Spain.}
\email{david.fisac-camara@cnrs.fr}
\author{Mingkun Liu}
\address{Mingkun Liu, DMATH, University of Luxembourg, Esch-sur-Alzette, Luxembourg.}
\email{mingkun.liu@math.univ-paris13.fr}


\thanks{DF is supported by the Luxembourg National Research Fund PRIDE17/1224660/GPS and acknowledges support by the FEDER/AEI/MICINN grant PID2021-125625NB-I00 “Estructuras y Desigualdades Geométricas Universales” and the AGAUR grant 2021-SGR-01015.\\
\indent ML is supported by the Luxembourg National Research Fund OPEN grant O19/13865598.}

\usepackage{setspace}

\usepackage{lipsum}

\begin{document}
\setstretch{1.15}

\begin{abstract}
We classify closed curves on a once-punctured torus with a single self-intersection from a combinatorial perspective.
We then use the classification to determine the number of closed curves with given word-length and with zero, one, and arbitrary self-intersections.
\end{abstract}

\maketitle

\part{Introduction}

Let $\Sigma_{1,1}$ be a (topological) torus with a puncture.
Free homotopy classes of oriented closed curves on $\Sigma_{1,1}$ naturally correspond to conjugacy classes of the fundamental group $\pi_1(\Sigma_{1,1})$.
We will tactically identify each free homotopy class of closed curves with its representatives, shifting freely between these two viewpoints, and henceforth referring to both as \emph{curves}.
Let $\gamma$ be a curve on $\Sigma_{1,1}$.
We say that $\gamma$ is \emph{essential} if it is not freely homotopic to a point or a loop around the puncture.
We denote by $\cC(\Sigma_{1,1})$ the set of free homotopy classes of essential closed curves on $\Sigma_{1,1}$, and $\cC^*(\Sigma_{1,1})$ the set of all free homotopy classes of curves on $\Sigma_{1,1}$.

\begin{figure}[ht]
    \begin{overpic}[width=.4\linewidth,keepaspectratio]{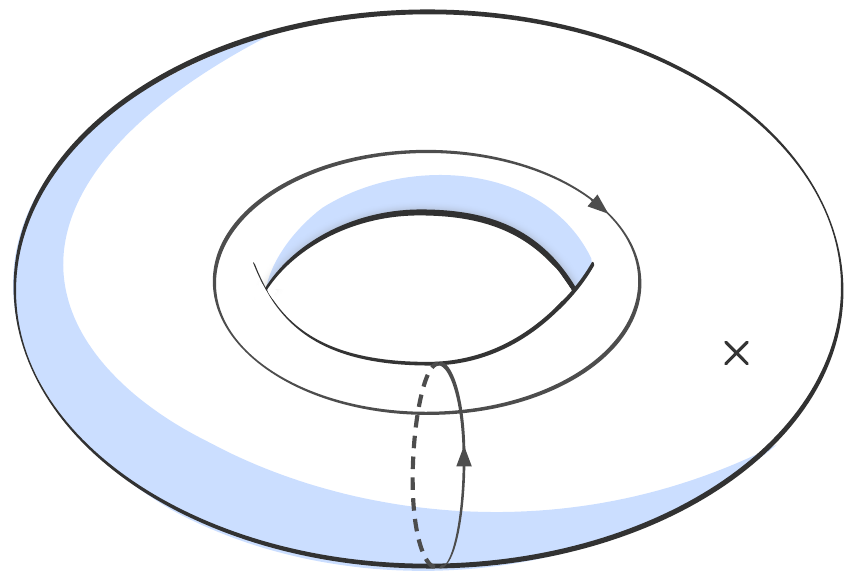}
    \put(56,12){$a$}
    \put(69,47){$b$}
    \end{overpic}
    \caption{A once-punctured torus}
    \label{fig:punctor}
\end{figure}

 We say that $\gamma$ is \emph{primitive} if its free homotopy class is not a proper power of another class, and we denote by $\Pp\cC(\Sigma_{1,1}) \subset \cC(\Sigma_{1,1})$ the set of primitive essential curves, and $\Pp\cC^*(\Sigma_{1,1})\subset \cC^*(\Sigma_{1,1})$ the set of all primitive curves.
The (geometric) intersection number of $\gamma$, denoted by $i(\gamma)$, is the minimum number of self-intersections among all curves within its free homotopy class:
\[
i(\gamma)
\coloneqq
\frac{1}{2} \min_{c}\{(t_1,t_2)\in \mathbb{S}^1\times \mathbb{S}^1\mid c(t_1)=c(t_2)\}
\]
where $c$ runs over the set of closed curves on $\Sigma_{1,1}$ in the free homotopy class $\gamma$.
Curves with $0$ self-intersection are commonly referred to as \emph{simple} curves.

The fundamental group $\pi_1(\Sigma_{1,1})$ is a free group of rank $2$.
By choosing a pair $(a,b)$ of generators for the fundamental group (see Figure~\ref{fig:punctor} for an example), we can identify $\pi_1(\Sigma_{1,1})$ with the free group $\mathbf{F}_2$ generated by $a$ and $b$.
A \emph{word} is a finite sequence $(x_i)_{i=1}^n$ such that $x_i \in \{ a,b,a^{-1},b^{-1} \}$ for all $1 \leq i \leq n$.
A \emph{circular word} can be seen as a word written on a circle, or more formally, an equivalence class of words where two words are equivalent if they differ by a circular shift.
A word $(x_i)_i^n$ is called \emph{reduced} if $x_{i+1} \neq x_{i}^{-1}$ for all $i = 1, \dots, n-1$, and $x_1 \neq x_n^{-1}$.
A circular word is \emph{reduced} if one (or all) of its representatives are reduced.
In this note, we consider only reduced (circular) words.
Words can be seen as elements in $\mathbf{F}_2$, and circular words can be seen as conjugacy classes in $\mathbf{F}_2$.
Curves in $\cC(\Sigma_{1,1})$ are in one-to-one correspondence with reduced circular words.
We say that a circular word is \emph{primitive} if none of its representatives is a proper power of another word.
Primitive curves correspond to primitive reduced circular words.

We say that a curve $\gamma\in\cC(\Sigma_{1,1})$ has \emph{word length} $n$, denoted by $\ell_\omega(\gamma)=n$, if it can be represented by a reduced word of $n$ letter.

\section{Results}

The results will appear in the order they will be proved during the article. The first one is a well-known result as a consequence of the homotopy-homology correspondence on the once-punctured torus. We will give a new proof of combinatorial nature, whose ingredients are useful later to prove a higher intersection case. 

\begin{theorem} \label{mainth1}
For any $L\in\ZZ_{\geq 4}$, we have
\[
| \{ \gamma \in \Pp\cC(\Sigma_{1,1})
\mid
i(\gamma) = 0, \, \ell_\omega(\gamma) = L \} |
=
4 \varphi(L),
\]
where $\varphi$ stands for Euler's totient function. 
\end{theorem}

\begin{corollary}
For $L\in\ZZ_{\geq 4}$, we have 
\[
| \{ \gamma \in \Pp\cC(\Sigma_{1,1})
\mid
i(\gamma) = 0, \, \ell_\omega(\gamma) \leq L \} |
=
4\Phi(L)+2,
\quad
\text{where}
\quad
\Phi(L)
\coloneqq
\sum_{n=1}^L \varphi(n),
    \]
and asymptotically,
\[
| \{ \gamma \in \mathcal{PC}(\Sigma_{1,1})
\mid
i(\gamma) = 0, \, \ell_\omega(\gamma) \leq L \} |
=
\frac{12}{\pi^2} \, L^2 + \bO(L(\log L)^\frac{2}{3}(\log\log L)^\frac{4}{3}).
\]
\end{corollary}
This estimate for the error term is the best known and is due to Walfisz, for further study on this see \cite[Theorem~1.3]{LemkeOliver-Soundararajan}.

A simple multicurve is a formal sum of pairwise non-homotopic disjoint simple curves with positive integer coefficients.
On $\Sigma_{1,1}$, the set of multicurves is nothing but $\cC(\Sigma_{1,1})$.
The last result allows us to count multicurves on $\Sigma_{1,1}$, which the authors found missing in the literature.
\begin{corollary}
For $L\in\ZZ_{\geq 4}$, we have 
\[
| \{ \gamma \in \mathcal{C}(\Sigma_{1,1})
\mid
i(\gamma) = 0, \, \ell_\omega(\gamma) = L \} |
=
4L
\]
and
\[
| \{ \gamma \in \mathcal{C}(\Sigma_{1,1})
\mid
i(\gamma) = 0, \, \ell_\omega(\gamma) \leq L \} |
=
2L^2 + 2L.
\]
\end{corollary}

The main result of the paper is the following. In the effort of counting non-necessarily simple curves on the one one-holed torus, we give a combinatorial classification of all curves with a single self-intersection, an analog theorem to Buser--Semmler's Theorem~\ref{thm:bus} for simple curves.

\begin{theorem}\label{thm:si1}
    A primitive curve in $\Pp\cC(\Sigma_{1,1})$ has self-intersection one if and only if, up to renaming the generators in $\{a,b,a^{-1},b^{-1}\}$, it can be written as one of the following:
    \begin{enumerate}
        \item $a^2b^2$, $aba^{-1}b$, $ab^{-1}a^{-1}b^2$, or
        \item $ab^{-1}a^{-1}b a^{n_1}b\cdots a^{n_k}b$, or $ab^{-1}a^{-1}b a^{-n_1}b\cdots a^{-n_k}b$, 
    where the words $a^{n_1}b\cdots a^{n_k}b$, and $a^{-n_1}b\cdots a^{-n_k}b$ are uniquely determined representatives of primitive simple curves, or 
        \item $a^{n_1}b\cdots a^{n_k}b$,
    where $[n_1,\hdots,n_k]$ satisfies that exists an $m\in\ZZ_{\geq1}$ such that for all $i\in\{1,\hdots, k\}$, $n_i\in\{m,m+1\}$ and it is a necklace with $2$-variation (see Definition \ref{def:2v}), or 
        \item $a^mba^{m+2}b$, for some $m\in\ZZ_{\geq1}$. 
    \end{enumerate}
\end{theorem}

The non-necessarily primitive case is just the sum of the number of primitive ones plus the number of primitive simple ones with half the length - as these will square to a single self-intersection. The classification allows us again to do the exact count for any $L$, using techniques from the proofs of the simple case. As in the case of simple closed curves, this counting result can already be found in the literature, more specifically in Moira Chas' Preprint \cite[Proposition~3.2]{Chas}. This new combinatorial proof is longer than the existing one due to the use of new combinatorial objects and the need to check further details.

\begin{theorem}
    \label{theorem:si1}
    There are $8$ primitve closed curves on $\Sigma_{1,1}$ of length $4$ with $1$ self-intersection.
    For any $L\in\ZZ_{>4}$, we have
    \[
    | \{
    \gamma \in \Pp\cC(\Sigma_{1,1})
    \mid
    i(\gamma) = 1, \, \ell_\omega(\gamma) = L
    \} |
    =
    \begin{cases}
        8 \, \varphi(L-4) & \text{if $L$ is odd,}\\
        8 \big( \varphi(L-4) + \varphi(L/2)/2 \big) & \text{if $L$ is even.}
    \end{cases}
    \]
\end{theorem}

\begin{corollary}
Under the same assumptions, the expressions sum up to
\[
| \{ \gamma \in \Pp\cC(\Sigma_{1,1})
\mid
i(\gamma) = 1, \, \ell_\omega(\gamma) \leq L
\} |
=
8 (\Phi(L-4) + \Phi(\floor{L/2}) / 2 ),
\]
    where  $\Phi(L)$ denotes the summation up to $L$ of Euler's totient function, and, asymptotically,
    \[
| \{ \gamma \in \Pp\cC(\Sigma_{1,1})
\mid
i(\gamma) = 1, \, \ell_\omega(\gamma) \leq L
\} |
\sim
\frac{27}{\pi^2} \, L^2.
\]
\end{corollary}

\begin{corollary}
We have
\[
\lim_{L \to \infty}
\frac{| \{ \gamma \in \Pp\cC(\Sigma_{1,1}) \mid i(\gamma) = 0,\, \ell_\omega(\gamma) \leq L \} |}{| \{ \gamma \in \Pp\cC(\Sigma_{1,1}) \mid i(\gamma) = 1,\, \ell_\omega(\gamma) \leq L \} |}
=
\frac{4}{9}.
\]
\end{corollary}
Here the word length $\ell_\omega$ can be replaced by the hyperbolic length induced by any complete hyperbolic metric on $\Sigma_{1,1}$ (see the section ``Related work''), and this corollary can be interpreted as follows:
on $\Sigma_{1,1}$, the probability that a random curve with at most one self-intersection has one self-intersection is $9/13$.

Lastly, with different methods, we count the curves of given word length without restriction on the self-intersection.

\begin{theorem} \label{mainth3}
There are $4$ primtive curves of length $1$, $8$ of length $2$, and for any $L \in \ZZ_{\geq 3}$, we have formula
\[
| \{ \gamma \in \Pp\cC^*(\Sigma_{1,1}) \mid \ell_\omega(\gamma) = L \} |
=
\frac{1}{L}\sum_{d|L}\mu(d)\,3^{L/d},
\]
where $\mu$ is the Möbius function. For not necessarily primitive curves, we have for any $L \in \ZZ_{\geq 1}$,
\[
| \{ \gamma \in \mathcal{C}^*(\Sigma_{1,1}) \mid \ell_\omega(\gamma) = L \} |
=
\frac{1}{L} \sum_{d \mid L} \varphi(d) \, 3^{L/d} + \frac{3+(-1)^L}{2}.
\]
\end{theorem}
\begin{corollary} \label{cor:3/2}
    We have
    \[
        | \{ \gamma \in \Pp\cC^*(\Sigma_{1,1}) \mid \ell_\omega(\gamma) = L \} |
        \sim
        \frac{3^L}{L},
        \qquad
        | \{ \gamma \in \Pp\cC^*(\Sigma_{1,1}) \mid \ell_\omega(\gamma) \leq L \} |
        \sim
        \frac{3}{2} \cdot \frac{3^L}{L}
    \]
    and the same holds true if $\Pp\cC^*$ is replaced by $\cC^*$.
\end{corollary}
\begin{remark} Hence, for $L\in\ZZ_{\geq1}$,
    \begin{align*}
     | \{ \gamma \in \Pp\cC^*(\Sigma_{1,1}) \mid \ell_\omega(\gamma)& = L \} |=\\ 
     &=|\{ \text{aperiodic necklaces with $L$ beads and }3\text{ colors}\}|+\delta_{\{1,2\}}(L),
     \end{align*} where $\delta_{\{1,2\}}(L)=1$ if $L\in\{1,2\}$ and vanishes otherwise, and
     \begin{align*} | \{ \gamma \in \mathcal{C}^*(\Sigma_{1,1}) \mid \ell_\omega(\gamma) = L \} |=|\{ \text{necklaces with $L$ beads and }3\text{ colors}\}|+\epsilon(L),
     \end{align*}
     where $\epsilon(L)=1$ if $L$ is odd and $\epsilon(L)=2$ if $L$ is even.

         Even when knowing this numerical coincidence, the authors could not find straightforwardly any natural bijection between these sets. However, it remains an open question, whose answer would possibly give automatically the counting on higher genus.
\end{remark}

\section{Related work}
Curve counting problems have been extensively studied, especially in the context of hyperbolic geometry.
More precisely, given a complete hyperbolic metric $X$ on a topological surface $\Sigma_{g,n}$ of genus $g$ with $n$ punctures, in each free homotopy class of essential curves, there exists a unique geodesic representative.
Therefore, rather than using the word length, we can also define the length of a free homotopy class $\gamma \in \cC(\Sigma_{g,n})$ by the length $\ell_X(\gamma)$ of its geodesic representative induced by $X$.
The famous prime number theorem for geodesics asserts that
\begin{equation} \label{eq:PNT}
    | \{ \gamma \in \Pp\cC(\Sigma_{g,n}) \mid \ell_X(\gamma) \leq L \} |
    \sim
    \frac{\mathrm{e}^L}{L}.
\end{equation}
This was an achievement initiated in the mid-20th century, and the names of Delsarte, Hejhal, Huber, Margulis, Selberg, Sarnak feature most strongly.
The estimate \eqref{eq:PNT} can be made effective, and the error terms are related to the Laplacian spectrum of $X$; see, for example, \cite[Section~5.4.2]{Bergeron}, \cite[Section~9.6]{Buser}.

One can also count curves under more topological constraints.
For example, one may restrict consideration to curves with no self-intersection, or a given number of self-intersections.
Efforts in this direction (see for example \cite{Rees, Birman-Series:simple, McShane-Rivin, Rivin}) reached their greatest heights with Mirzakhani's groundbreaking work \cite{Mirzakhani:first}.
As a consequence of her findings, there exist explicit constants $C_{g,n} > 0$ depending only on $g$ and $n$ and $B_X > 0$ depending only on the hyperbolic metric $X$ such that
\begin{equation} \label{eq:i=0}
    | \{ \gamma \in \Pp\cC(\Sigma_{g,n}) \mid i(\gamma) = 0, \, \ell_X(\gamma) \leq L \}|
    \sim
    C_{g,n} \cdot B_X \cdot L^{6g-6+2n}.
\end{equation}
This result was extended to the non-simple case by Mirzakhani twelve years later \cite{Mirzakhani:last}: for any $k \in \ZZ_{\geq 0}$ there exists a constant $C_{g,n;k}$, such that \eqref{eq:i=0} remains valid when replacing ``$i(\gamma)=0$'' by ``$i(\gamma)=k$'' and $C_{g,n}$ by $C_{g,n;k}$.

It turns out that the length function $\ell_X \colon \Pp\cC(\Sigma_{g,n}) \to \mathbb{R}$ defined by the hyperbolic metric $X$ is not essential: Erlandson and Souto \cite[Theorem~1.2]{ErlandssonSouto} proved that a very similar result holds for any ``positive, continuous and homogeneous function on the space of geodesic currents on $\Sigma_{g,n}$''.
More precisely, for any such a function $\ell$, which can be taken as $\ell_\omega$ (see \cite{Erlandsson}) or $\ell_X$, there exist constants $C_{g,n;k}$ depending only on $g,n$ and $k$, and $B_{\ell}$ depending only on the function $\ell$, such that
\begin{equation} \label{eq:i=k}
    | \{ \gamma \in \Pp\cC(\Sigma_{g,n}) \mid i(\gamma) = k, \, \ell(\gamma) \leq L \}|
    \sim
    C_{g,n;k} \cdot B_{\ell} \cdot L^{6g-6+2n}.
\end{equation}
As a result, the following limit
\[
    \lim_{L \to \infty}
    \frac{| \{ \gamma \in \Pp\cC(\Sigma_{g,n}) \mid i(\gamma) = 0, \, \ell(\gamma) \leq L \} |}{| \{ \gamma \in \Pp\cC(\Sigma_{g,n}) \mid i(\gamma) = 1, \, \ell(\gamma) \leq L \} |}
\]
exists and does not depend on the choice of length function $\ell$.
Finally, let us mention that the estimate \eqref{eq:i=0} can also be made effective; see \cite{McShane-Rivin, Eskin-Mirzakhani-Mohammadi, AranaHerrera}.

In terms of counting curves with a given word-length and self-intersection, there are many works by Chas, Phillips, Lalley, and McMullen; see \cite{Chas-Phillips, Chas-Lalley, Chas:2015, Chas-McMullen-Phillips}. Many bounds have been found for the general cases and closed formulas for given length-intersection difference. The main question in this regard is to classify all words of a given self-intersection and find a closed formula for 
\[
 | \{ \gamma\in \Pp\cC (\Sigma_{1,1})\mid i(\gamma)=k, \ell_\omega(\gamma)=L\} |.
\]

\subsection*{Organization of the paper}
Section \ref{section:simple} of the paper will first introduce in our notation the characterization of simple closed curves by \cite{BuserSemmler}, study its combinatorial rigidity, and then prove Theorem \ref{mainth1} with it. Section \ref{section:si1} will start by doing the analogous characterization of curves, but this time targeting self-intersection one, hence proving Theorem \ref{thm:si1}. The last part of the section will prove Theorem \ref{theorem:si1} by using the combinatorics already studied in Section \ref{section:simple}. Finally, in Section \ref{section:all} we use different methods to prove Theorem \ref{mainth3}, which will be analytic combinatorics, giving the broadest image of our counting.

\subsection*{Acknowledgements}
The authors would like to thank Hugo Parlier for his very helpful comments and constant interest.
Also for enlightening conversations with Pierre Arnoux, Valérie Berthé, Viveka Erlandsson, and Alexey Balitskiy (\textcyr{спасибочки Лёшенька!}).

\part{Simple curves}\label{section:simple}
In this section, we study simple curves, namely closed curves without self-intersection, on the once-punctured torus $\Sigma_{1,1}$.
As an application of the techniques we are about to develop, we provide an alternative proof for Theorem~\ref{mainth1}.

\section{Curves as necklaces}
As we will soon see, the following seemingly unrelated definition will simplify our discussion.
\begin{definition}[Necklace]
    A \emph{necklace of integers}, or simply a \emph{necklace}, which can be seen as a (finite) sequence of positive numbers written on a circle, is an equivalence class of a finite sequence of positive integers $(n_i)_i$ where two sequences are equivalent if they differ by a circular shift.
\end{definition}
We use ``$(\dots)$'' to denote sequences, and use ``$[\dots]$'' to denote necklaces.
For instance, $(1,2,3)$, $(2,3,1)$, and $(3,1,2)$ are different sequences but represent the same necklace $[1,2,3]$.

\begin{definition}[Small variation] \label{def:smv}
    A necklace $[n_i]_i$ has \emph{small variation} if, for any $s \in \ZZ_{\geq 1}$, sums of $s$ consecutive elements never differ by more than $\pm 1$.
    In symbols, this means
    \begin{equation} \label{eq:smaVa}
    \bigg| \sum_{j=1}^s n_{i_1+j} - \sum_{j=1}^s n_{i_2+j} \bigg| \leq 1
    \end{equation}
    for all $i_1,i_2$, and indices are taken modulo $k$.
\end{definition}

For example,
the necklaces $[5, 5, 5, 5]$, $[5, 5, 5, 4]$, and $[5, 4, 5, 4]$ have small variation, but $[5, 5, 3]$, $[5, 5, 4, 4]$, and $[5, 5, 4, 5, 5, 5, 4, 5, 4]$ do not.

The following definitions are standard.
Let $w$ be a necklace and $m \in \ZZ$.
We denote by $|w|_m$ the number of occurrences of $m$ in $w$.
A sequence is called a \emph{block} of $w$, if it can be found as a contiguous subsequence within a sequential representation of $w$.
For example, $(1,2)$ and $(4,1)$ are blockss of $[1,2,3,4]$, but $(1,4)$ is not.
A \emph{run} is a constant block that is not properly contained in any constant block.
For instance, there are $3$ runs of $2$ in $[2,1,2,1,2,1,2,2]$ (two of length $1$ and one of length $3$).

\begin{remark} \label{rem:byDef}
    Let $w = [n_i]_i$ be a non-constant necklace.
    If \eqref{eq:smaVa} holds for $s=1$, then there exists $m \in \ZZ$ such that $n_i \in \{ m, m+1 \}$ for all index $i$.
    If \eqref{eq:smaVa} holds for $s=2$ too, then $m$ (resp. $m+1$) cannot appear consecutively if $|w|_m \leq |w|_{m+1}$  (resp. $|w|_{m+1} \leq |w|_m$).
\end{remark}

\begin{lemma} \label{lemma:gapsize}
Let $m, x, y \geq 1$, and let $w$ be a necklace with small variation that contains exactly $x$ occurrences of the number $m$, and $y$ occurrences of the number $m+1$.
Write $q \coloneqq \max(x,y)/\min(x,y)$.
If $x \geq y$, then all runs of $m$ in $w$ have size $\floor*{q}$ or $\ceil*{q}$, and all runs of $m+1$ in $w$ have size $1$.
If $x \leq y$, then all runs of $m+1$ have size $\floor*{q}$ or $\ceil*{q}$, and all runs of $m$ have size $1$.
\end{lemma}
\begin{proof}
Without loss of generality, we assume that $x \geq y$.
It follows from \eqref{eq:smaVa} by taking $s = 2$ that all runs of $m+1$ have size $1$.
If there exists a run of $m$ of size at most $\floor{q}-1$, then there is a run of $m$ of size at least $\floor{q}+1$, and it follows from \eqref{eq:smaVa} by taking $s = \floor{q}+1$ the necklace does not have small variation.
Similarly, if there exists a run of $m$ of size at least $\ceil{q}+1$, then the necklace does not have small variation.
\end{proof}

One of the main reasons why we are interested in necklaces with small variation is the following.

\begin{theorem}[{\cite[Theorem~6.2]{BuserSemmler}}] \label{thm:bus}
    Every simple closed curve on $\Sigma_{1,1}$ can be represented, after suitably renaming the generators, by one of the following words:
    \begin{enumerate}
        \item \label{type1}
            $a$,

        \item \label{type2}
            $aba^{-1}b^{-1}$,

        \item \label{type}
            $ab^{n_1} a b^{n_2} \cdots a b^{n_r}$, where $[n_1, \dots, n_r]$ has small variation.
    \end{enumerate}
    Conversely, each of these words is homotopic to a power of a simple closed curve.
\end{theorem}

\begin{remark}\label{rmk:stur}
    Finite words can be seen as periodic infinite words.
    The small variation condition on the exponents is nothing but the balance condition on words that defines finite Sturmian words in the literature (see for example \cite{Vuillon, Glen-Justin}).
\end{remark}
We say that a simple closed curve on $\Sigma_{1,1}$ has \emph{general type} if it falls into the third case described in Theorem~\ref{thm:bus}.
For such a curve, we associate it with the necklace $[n_1, \dots, n_r]$, which we will refer to as its \emph{exponent necklace}.
See Figure~\ref{fig:circW} and \ref{fig:exp} for an example.

Write $M(L)$ for the set of closed multicurves of general type on $\Sigma_{1,1}$ of length $L$, and $N(L)$ for the set of necklaces $[n_1, \dots, n_r]$ such that $r + \sum_i n_i = L$.
\begin{corollary} \label{cor:necurve}
    For $L$ odd, $M(L)$ and $N(L)$ are in eight-to-one correspondence. For $L$ even, we have $|M(L)| = 8|N(L)| - 4$.
\end{corollary}
\begin{proof}
    First, it is not hard to check that a power of a necklace (the square of $[1,2,3]$ is $[1,2,3,1,2,3]$) with small variation has small variation.
    Thus every power of a (primitive) simple curve of general type is of general type.
    The eight-to-one correspondence in the case where $L$ is odd arises from the $8$ possible renamings of the generators.
    This correspondence breaks down when $L$ is even due to the fact that the necklace $[1,\dots, 1]$ corresponds to only $4$ curves (for example, $ab \cdots ab$ and $ba \cdots ba$ give the same curve).
\end{proof}

Therefore, counting simple closed curves on $\Sigma_{1,1}$ with respect to the word length can be boiled down to the enumeration of necklaces with small variation.
In the next section, we will prove a rigidity result for such necklaces, which allows us to further reduce our curve counting problem to a problem of lattice point counting.

\section{Necklaces with small variation}

The following is the main result of the section. That is, from the periodic Sturmian word perspective, the equivalence in definitions as balanced words mentioned before in Remark \ref{rmk:stur} and as cutting sequences on the integer grid for an infinite ray starting at the origin with rational slope.
This fact is well-known to the experts, however, as far as the knowledge of the authors gets, there seems not to be any good reference on this fact and a proper bijection. In what follows we give a self-contained elementary proof. 

\begin{proposition} \label{prop:EU} 
    Given $m\in\ZZ_{\geq1}$, and x,y $\in\ZZ_{\geq 0}$,
    there exists a unique necklace with small variation that contains exactly $x$ occurrences of the number $m$, and $y$ occurrences of the number $m+1$.
\end{proposition}

The proof proceeds by induction, where the operations on necklaces we are defining now play an important role.

Let us start with the following automorphisms of $\mathbf{F}_2$ defined for any $m \in \ZZ_{\geq 1}$ by
\begin{equation} \label{eq:auto}
    \alpha_m :
    \begin{array}{ll}
        a^m b & \mapsto b,\\
        a^{m+1} b & \mapsto a,
    \end{array}
    \qquad
    \tilde{\alpha}_m :
    \begin{array}{ll}
        a^m b & \mapsto a,\\
        a^{m+1} b & \mapsto b.
    \end{array}
\end{equation}
They are well-defined because $a^m b$ and $a^{m+1} b$ form a basis of $\mathbf{F}_2$.

The following elementary lemma will be important for our purposes.
\begin{lemma} \label{lem:red}
    For any word $w$ and any $m \in \ZZ_{\geq 1}$, the curves represented by $w$, $\alpha_m(w)$, $\tilde{\alpha}_m(w)$, $\alpha_m^{-1}(w)$, and $\tilde{\alpha}_m^{-1}(w)$ have the same self-intersection number.
\end{lemma}
\begin{proof}
The automorphisms $\alpha_m$ and $\tilde{\alpha}_m$ preserve the set of conjugacy classes that correspond to the puncture (the conjugacy class of $aba^{-1}b^{-1}$ and the conjugacy class of its inverse $bab^{-1}a^{-1}$).
Now the Dehn--Nielsen--Baer theorem implies that the actions of $\alpha_m$ and $\tilde{\alpha}_m$ on
conjugacy classes of $\mathbf{F}_2 \simeq \pi_{1}(\Sigma_{1,1})$ (circular words and free homotopy classes of curves) are induced by some self-homeomorphisms of $\Sigma_{1,1}$.
The lemma follows.
\end{proof}

Let $w$ be a necklace with small variation.
Recall that we write $|w|_m$ for the number of occurrences of $m$ in $w$.
We define a new necklace $A(w)$ as follows.
By Remark~\ref{rem:byDef}, there exists $m \in \ZZ_{\geq 1}$ such that $n_i \in \{ m, m+1 \}$ for all $i$.
If $|w|_m \leq |w|_{m+1}$ (resp.\ $|w|_{m+1} \leq |w|_m$),
we define $A(w)$ to be the necklace obtained by removing all the $m$'s (resp.~$(m+1)$'s) from $w$ and replacing each run of $m+1$ (resp.~$m$) by the length of the run.
For example, $A[4,5,5,4,5,5,5,4,5,5] = [2,3,2]$ and $A[4,5,5,5,4,5,4,5,5] = [3,1,2]$.
Note that if $|w|_m = |w|_{m+1}$, then $A(w)$ is a constant necklace $[1, \dots, 1]$ of size $|w|_m=|w|_{m+1}$.

This operation can be defined equivalently as follows.
Let $w = [n_1, \dots, n_r]$, and consider the word $\omega = a^{n_1} b \cdots a^{n_r} b$.
If $|w|_m \leq |w|_{m+1}$ (resp.\ $|w|_{m+1} \leq |w|_m$), then $A(w)$ is the exponent necklace of $\alpha_m(\omega)$ (resp.\ $\tilde{\alpha}_m(\omega)$).

With certain supplementary information, the map $A$ can be reversed.
We define $B_m(w)$ (resp.\ $\tilde{B}_m(w)$) to be the necklace obtained by replacing each $n_i$ by a run of $m+1$ (resp.\ $m$) of length $n_i$ and inserting a $m$ (resp.\ $m+1$) between every two consecutive runs of $m+1$ (resp.\ $m$).
For example, $B_{3}[1,1,3,2] = [4,3,4,3,4,4,4,3,4,4,3]$.
In other words, if $|w|_m \leq |w|_{m+1}$ (resp.\ $|w|_{m+1} \leq |w|_m$), then $B_m(w)$ is the exponent necklace of $\alpha_m^{-1}(\omega)$ (resp.\ $\tilde{\alpha}_m^{-1}(\omega)$).
Now, by construction, if $w$ is a necklace with small variation, then $B_{m}(A(w)) = w$ if $|w|_m \leq |w|_{m+1}$, and $\tilde{B}_m(A(w)) = w$ if $|w|_m \geq |w|_{m+1}$.

\begin{figure}[ht]
\centering
\begin{subfigure}[b]{0.32\textwidth}
\centering
\includegraphics[width=.8\textwidth]{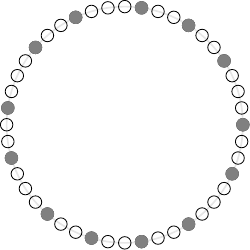}
\caption{A circular word}
\label{fig:circW}
\end{subfigure}
\begin{subfigure}[b]{0.32\textwidth}
\centering
\includegraphics[width=.8\textwidth]{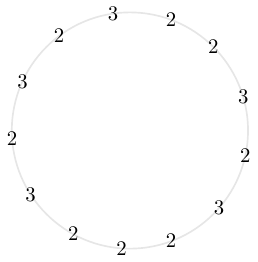}
\caption{Its exponent necklace}
\label{fig:exp}
\end{subfigure}
\begin{subfigure}[b]{0.32\textwidth}
\centering
\includegraphics[width=.8\textwidth]{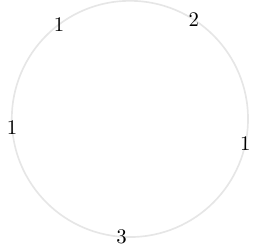}
\caption{Image of $A$}
\label{fig:Aexp}
\label{mingkun}
\end{subfigure}
\end{figure}

A necklace $w = [n_i]_i$ is said to have \emph{profile} $(m,x,y)$, if $n_i \in \{ m, m+1 \}$ for all $i$, $|w|_m = x$, and $|w|_{m+1} = y$.

Note now that by Lemma \ref{lemma:gapsize}, if $\min(x,y)$ divides $\max(x,y)$, then Proposition \ref{prop:EU} follows immediately.

\begin{lemma} \label{lem:min}
    Let $m,x,y \in \ZZ_{\geq 1}$, and $w_1, w_2$ be two necklaces with the same profile $(m,x,y)$.
    Then $A(w_1)$ and $A(w_2)$ have the same profile $(m', x', y')$ where
    \begin{equation} \label{eq:Gmxy}
        \begin{split}
            m' & =  \lfloor \max(x, y)/\min(x, y) \rfloor, \\
            x' & =  \min(x, y) - \max(x, y) + \min(x, y) \lfloor \max(x, y)/\min(x, y) \rfloor, \\
            y' & =  \max(x, y) - \min(x, y) \lfloor \max(x, y)/\min(x, y) \rfloor,
        \end{split}
    \end{equation}
    which depends only on $(x, y)$.
    Moreover, we have $0 \leq \min(x', y') < \min(x, y)$.
\end{lemma}
\begin{proof}
    Let $(m', x', y')$ be the profile of $A(w_1)$. Note that $m'$ is determined by Lemma~\ref{lemma:gapsize}, the total size of $A(w_1)$ is given by construction by $\min(x,y)$ and the sum of all the elements in $A(w_1)$ is given by construction by $\max(x,y)$. Hence, $(m', x', y')$ is determined by the system
    \[
        m' = \lfloor \max(x, y)/\min(x, y) \rfloor,
        \quad
        x' + y' = \min(x, y),
        \quad
        x' m' + y'(m'+1) = \max(x, y).
    \]
    These equations have a unique solution given by \eqref{eq:Gmxy}.
    A direct computation shows that
    \[
        x'
        \leq
        \min(x, y),
        \qquad
        y'
        \leq
        \min(x, y)
    \]
    and $x' = \min(x, y)$ if and only if $\min(x, y)$ divides $\max(x, y)$, in which case $y' = 0$ and hence $y' < \min(x, y)$.
    The lemma follows.
\end{proof}
Now, we are ready to the main result of the section.
\begin{proof}[Proof of Proposition~\ref{prop:EU}]
    The assertion is evident if $\min(x,y) = 0$.
    Assume $x, y \geq 1$.
    We proceed by induction on $\min(x,y)$.
    Suppose that the proposition holds for all triples $(m,x,y) \in \ZZ_{\geq 1}^3$ where $\min(x,y) \leq k$.
    We will prove that the proposition holds for all $(m,x,y) \in \ZZ_{\geq 1}^3$ such that $\min(x,y) = k+1$.
    
    By Lemma~\ref{lem:min} and the induction hypothesis, there exist a unique necklace $w'$ with profile $(m', x', y')$, defined by \eqref{eq:Gmxy}.
    By Lemma~\ref{lem:red}, if $x \geq y$ (resp.\ $x \leq y$), then $B_m(w')$ represents a simple curve with profile $(m,x,y)$.
    This proves the existence.
    The uniqueness follows from Lemma~\ref{lem:min} the fact that $A$ is injective (by Lemma~\ref{lem:red}).
    This completes the proof.
\end{proof}

\section{Counting}

We will start by counting the curves of general type, which means it can be represented by a word of the form $ab^{n_1}ab^{n_2}\cdots ab^{n_r}$, and then at the rest of the curves. Note that such a curve is determined by its exponent necklace $w=[n_1,\hdots, n_r]$.

Now, if $n_i\in\{m,m+1\}$, the curve $\gamma=ab^{n_1}ab^{n_2}\cdots ab^{n_r}$ has word length $|w|_m(m+1)+|w|_{m+1}(m+2)$.

\begin{proposition}\label{prop:4l1}
    For any positive integer $L$, there are exactly $4(L-1)$ simple multicurves of general type and length $L$.
\end{proposition}

To prove so we first prove a proposition on our Diophantine equations.

\begin{proposition} \label{prop:=}
    Consider the equation
    \begin{equation} \label{eq:K}
        x(m+1) + y(m+2) = L
    \end{equation}
    where $L \in \ZZ_{\geq 1}$ is given, and $x,m \in \ZZ_{\geq 1}$, $y \in \ZZ_{\geq 0}$ are unknown.
    Then, there are exactly $L/2$ solutions $(x,y,m)$ if $L$ is even, and $(L-1)/2$ solutions if $L$ is odd.
\end{proposition}

\begin{proof}
    Let $S(L)$ denote the set of triples $(m,x,y) \in \ZZ_{\geq 1}^2 \times \ZZ_{\geq 0}$ satisfying \eqref{eq:K}.
    Since $|S(2)| = 1$, it suffices to prove that for any $L \in \ZZ_{\geq 2}$, $|S(L+1)| = |S(L)|$ if $L$ is even, and $|S(L+1)| = |S(L)| + 1$ if $L$ is odd.
    The strategy is to show that there exists a map $\Lambda_+ \colon S(L) \to S(L+1)$ such that, when $L$ is even, $\Lambda_+$ is a bijection, and when $L$ is odd, $\Lambda_+$ is an injection and $|\Lambda_+(S(L))| = |S(L+1)|-1$.

    Let us define $\Lambda_+ \colon \ZZ^3 \to \ZZ^3$ by the formula
    \[
        \Lambda_+(x,y,m)
        \coloneqq
        \begin{cases}
            (x-1,y+1,m) & \text{if } x \geq 2, \\
            (y+1,0,m+1) & \text{if } x = 1.
        \end{cases}
    \]
    A direct computation shows that $\Lambda_+(S(L)) \subset S(L+1)$.
    Now consider the map $\Lambda_- \colon \ZZ^3 \to \ZZ^3$ defined by
    \[
        \Lambda_-(x,y,m)
        \coloneqq
        \begin{cases}
            (x+1,y-1,m) & \text{if } y \geq 1, \\
            (1,x-1,m-1) & \text{if } y = 0.
        \end{cases}
    \]
    Again, a straightforward computation shows that $\Lambda_-(S(L+1)) \subset S(L)$, and the composition $\Lambda_- \circ \Lambda_+ \colon S(L) \to S(L)$ is the identity.
    In particular, the restriction of $\Lambda_+$ on $S(L)$ is an injection.
    Now, observe that $(x,y,m) \in S(L+1)$ and $\Lambda_-(x,y,m) \notin S(L)$ if and only if $(y,m) = (0,1)$,
    and there exists $x \in \ZZ_{\geq 1}$ such that $(x,0,1) \in S(L+1)$ if and only if $L+1$ is even.
    This completes the proof.
\end{proof}
\begin{proof}[Proof of Proposition \ref{prop:4l1}]
    This follows immediately from Corollary~\ref{cor:necurve}, Proposition~\ref{prop:EU} and \ref{prop:=}.
\end{proof}

Now we are ready to prove the simple curve counting result.
\begin{proof}[Proof of Theorem~\ref{mainth1}]
Define 
\begin{equation} \label{eq:M}
    M(n)
    \coloneqq
    | \{ \gamma \in \cC(\Sigma_{1,1}) \mid i(\gamma) = 0, \, \gamma \text{ is of general type},\, \ell_\omega(\gamma) = n \} |.
\end{equation}
By Proposition \ref{prop:4l1}, $m(n) = 4(n-1)$. Define also
\[
    P(n)
    \coloneqq
    | \{ \gamma \in \Pp\cC(\Sigma_{1,1}) \mid i(\gamma) = 0, \, \gamma \text{ is of general type},\, \ell_\omega(\gamma) = n \} |.
\]
By definition, $M(n) = \sum_{d \mid n} P(d)$. Applying the Möbius inversion formula we obtain
\[
    P(n)
    =
    \sum_{d \mid n} \mu(d) \, M(n/d)
\]
where $\mu(d)$ stands for the Möbius function. Using the arithmetic identities
\[
    \sum_{d \mid n} \frac{\mu(d)}{d}
    =
    \frac{\varphi(n)}{n},
    \qquad
    \sum_{d \mid n} \mu(d)
    =
    \begin{cases}
        1 & \text{if } n = 1, \\
        0 & \text{if } n \geq 2,
    \end{cases}
\]
where $\varphi$ is Euler's totient function,
we obtain for $n\geq1$,
\[
P(n)=4\varphi(n)-4\delta_{\{1\}}(n),
\]
where $\delta_{\{1\}}(n)=0$ for $n>1$ and $\delta_{\{1\}}(1)=1$.
After summing rewrites as:
\[
    \sum_{n=1}^{L} P(n)
    =
    4\Phi(L) - 4.
\]
Adding the four essential simple curves of non-general type: $a$, $b$, $a^{-1},b^{-1}$, the first part of the theorem follows.

For the second part, by Proposition \ref{prop:4l1}, we know that the number of powers of simple closed curves of general type with word length exactly $L$ is $M(L)= 4(L-1)$. Hence,
    \[
        \sum_{n=1}^{L} M(n)
        = \sum_{n=1}^L 4(n-1)=2(L^2-L).
\]
    Adding the words of type $a^k$, $b^k$, $a^{-k}$, $b^{-k}$ for $k=1,\hdots,L$, the theorem follows.
\end{proof}

\part{Self-intersection one} \label{section:si1}
In this section, we classify self-intersection one curves and derive a formula for the number of such curves of a given length.

\section{Characterization}
The aim of this section is to prove Theorem~\ref{thm:si1}.

One of the main tools for the section will be the algorithm to determine self-intersection proposed by Cohen and Lustig \cite{Cohen-Lustig} built on the work of Birman and Series \cite{Birman-Series}.
We will shortly introduce our notation on it right now.

Let $\omega = x_1 \cdots x_n$, where $x_i \in \{ a, b, a^{-1}, b^{-1} \}$ be a reduced word.
Consider all of its circular shifts $\{ \omega_i \}_{i=1}^{n}$ where $\omega_i = x_i x_{i+1} \cdots x_n x_1 x_2 \cdots x_{i-1}$ and the cyclic lexicographic ordering induced by $a<b<a^{-1}<b^{-1}$: two words $\alpha = x_{1} \cdots x_{n}$ and $\beta = y_1 \cdots y_n$ satisfy $\alpha < \beta$ if either $x_{1} < y_{1}$, or for some $1 < t \leq n$, we have $x_{i} = y_{i}$ for $1 \leq i \leq t-1$, and $x_{t} < y_{t}$ under the new ordering obtained by cyclically shifting the original one until it starts with $x_{t-1}^{-1}$.

For example, $baab^{-1} < baab$.
We will call $(i,j) \in \{1, \dots, n \}^2$ a \emph{linking pair} if $\omega_i<\omega_j<\omega_i^{-1}<\omega_j^{-1}$ or $\omega_i<\omega_j^{-1}<\omega_i^{-1}<\omega_j$.
Consider the set of linking pairs with the equivalence relation induced by $(i,j)\sim (i+1,j+1)$ if $\omega_i$ and $\omega_j$ start with the same letter and same sign, and $(i+1,j)\sim (i,j+1)$ if $\omega_i$ and $\omega_j$ start with the same letter and opposite sign.
Cohen and Lustig proved that the self-intersection number of the curve represented by $\omega$ equals the number of equivalence classes of linking pairs found this way.
We refer the reader to \cite{Cohen-Lustig} for more details.

\begin{proposition} \label{prop:si1case2}
Let $\omega = a^{n_1} b^{m_1} \cdots a^{n_k} b^{n_k}$ be a reduced word representing a curve with a single self-intersection.
If there exist $i, j \in \{1, \dots, k \}$ such that $|n_i|, |m_j| \geq 2$, then up to renaming of the generators and circular shift, the only possible $\omega$ is $a^2 b^2$.
\end{proposition}

\begin{proof}
Take a word that writes reduced as $\omega_1 a^{n_1} \omega_2 b^{n_2} \omega_3$ with $n_1=n_2=2$ and $\omega_1,\hdots,\omega_3$ being subwords, at least one of them being non-empty.
This word is the image by the ``cross-corner surgery'' described in \cite{Chas-Phillips} of a word $\omega_1^{-1}ab^{-1}\omega_2^{-1}a^{-1}b\omega_3$. By \cite[Proposition~2.2]{Chas-Phillips} this surgery increases self-intersection by at least one, and by Theorem \ref{thm:bus}, the initial word did not represent a simple curve, hence, self-intersection of the image is at least two. Same proceeding applies for $n_i=-2$ just by switching $a$ for $a^{-1}$, $b$ for $b^{-1}$ or both.
\end{proof}

Hence, from now on we can assume that one of the generators has only exponents $\{-1,1\}$. Following then,

\begin{proposition}\label{prop:si1case1}
Let $\omega$ be a word reprsenting a curve with a single self-intersection of the form $a^{n_1} b^{\epsilon_1} \cdots a^{n_k} b^{\epsilon_k}$ where $\epsilon_i \in \{ -1, 1 \}$ for all $i$ and $m_i m_j = -1$ for some $i,j$.
Then, $\omega$ is, up to renaming the generators and circular shift, of the form
    \begin{equation} \label{eq:3.2}
    a b^{-1} a^{-1} b \cdot a^{m_1} b \cdots a^{m_r} b
    \qquad \text{or} \qquad
    a b^{-1} a^{-1} b \cdot a^{-m_1} b \cdots a^{-m_r} b
    \end{equation}
    where $a^{m_1}b \cdots a^{m_r}b$
    represents a primtive simple curve,
    with the short exceptional cases $ab^{-1}ab,b^{-1}a^{-1}ba^2$.
    Conversely, every word representing a primitive simple curve of the form $a^{m_1} b \cdots a^{m_r} b$ can be circular shifted to a unique word $\sigma$ such that $a b^{-1} a^{-1} b \cdot \sigma$ represents a curve with a single self-intersection. 
\end{proposition}

\begin{proof}
    We will characterize all possible changes of sign.
    Start by considering a reduced word of the form $a b^{-1} a^i b a \omega$, with $i\geq1$ and $\omega$ being a subword. By applying a surgery as in Figure \ref{fig:ieq2}, 
\begin{figure}[ht]
\centering
\begin{subfigure}[b]{0.45\textwidth}
\centering
\begin{overpic}[width=.75\textwidth]{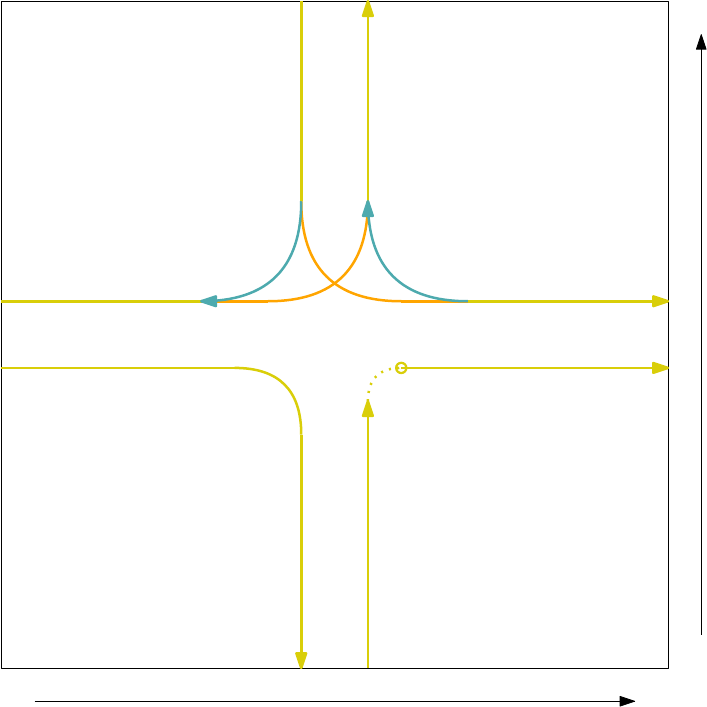}
\put(0,-1){$a$}
\put(98,3){$b$}
\end{overpic}
\caption{$i=1$}
\label{fig:ieq1}
\end{subfigure}
\begin{subfigure}[b]{0.45\textwidth}
\centering
\begin{overpic}[width=.75\textwidth]{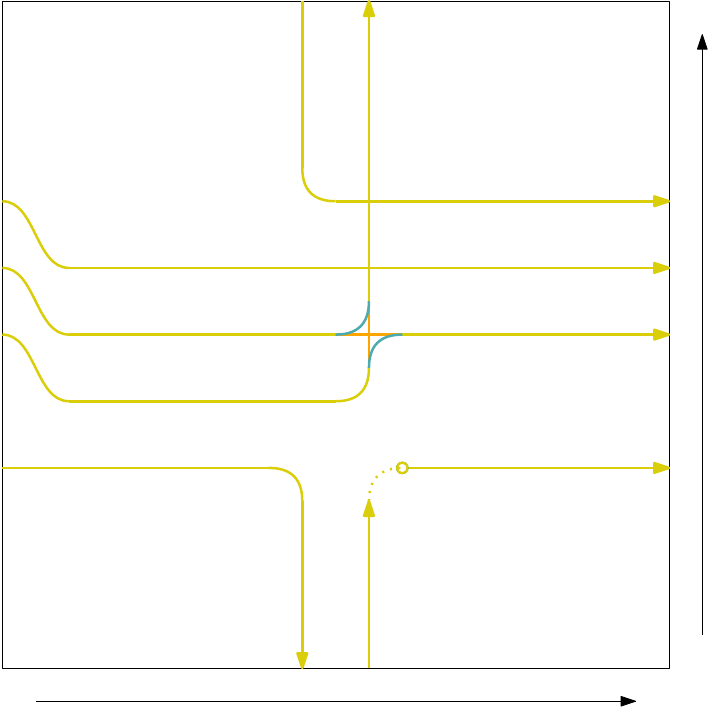}
\put(0,-1){$a$}
\put(98,3){$b$}
\end{overpic}
\caption{$i \geq 2$}
\label{fig:ieq2}
\end{subfigure}
\caption{Surgery on a word of the form $a b^{-1} a^i b a\omega$: remove the orange part and add the blue part.}
\end{figure}
if $i\geq2$, the word will lose one self-intersection and become $ab^{-1}a^{i-1}ba\omega$ (coming in the linking pair notation as losing the linking pair given by the cyclic shifts of the word: $\omega_1=ba\omega ab^{-1}a^i$ and  $\omega_2=aba\omega ab^{-1}a^{i-1}$). And, for $i=1$, the word also loses a self-intersection and becomes $a b^{-1} a^{-1} b \omega$ (in linking pairs notation loses the linking pair given by the two cyclic shifts of the word: $\omega_1=ab^{-1}a^{i}ba\omega$ and $\omega_2=aba\omega ab^{-1}a^{i-1}$). Now, by the classification of simple closed curves, this word is simple if and only if $\omega=\emptyset$, hence we find that the only candidate of the form  $ab^{-1}a^iba\omega$, with $i\geq1$, for self-intersection one is $ab^{-1}aba$, which indeed can be checked by the algorithm in \cite{Cohen-Lustig} to have a single self-intersection and will correspond to one of the exceptional short cases.

    For the rest of the cases, one can assume that when there is a change of sign in $b$, there is also a change of sign in $a$. In particular, since renaming all $a$ for $a^{-1}$ (and respectively with $b$) does not change the self-intersection number, we will consider words of the form $ab^{-1}a^{-i}b\omega$, with $i\geq2$, and $\omega$ being a subword.
     
\begin{figure}[ht]
    \begin{overpic}[width=0.35\linewidth,keepaspectratio]{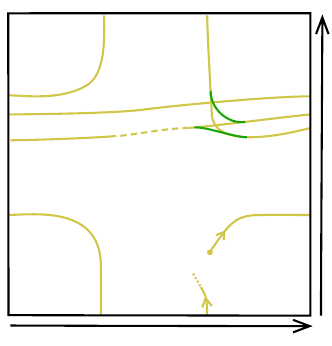}
    \put(46,0.5){$a$}
    \put(94,50){$b$}
    \end{overpic}
    \caption{Surgery on a word of the form $ab^{-1}a^{-i}b\omega$ with $i\geq2$}
    \label{fig:cut2}
\end{figure}
    Again, by applying a surgery as in Figure \ref{fig:cut2}, if $i\geq2$, the word will lose one self-intersection and become $ab^{-1}a^{-i+1}ba\omega$ (corresponding to losing the linking pair coming from the permutations of the word $\omega_1=b\omega ab^{-1}a^{-i}$ and $\omega_2=a^{-1}b\omega ab^{-1}a^{-i+1}$). Once again, by the classification of simple words, if $\omega\neq\emptyset$ this will never be simple and hence the initial word does not have a single self-intersection, whilst if $\omega=\emptyset$, this will be simple if and only if $i=2$, in which case the only possible candidate for a single self-intersection is the word $ab^{-1}a^{-2}b$, which can be checked by the algorithm in \cite{Cohen-Lustig} to be indeed of self-intersection one, giving the other short exceptional case.

    Hence, we can restrict to the case where the changes of sign come from subwords of the kind $ab^{-1}a^{-1}b$ (up to renaming $a$ for $a^{-1}$ or $b$ for $b^{-1}$). Hence, take a reduced word $ab^{-1}a^{-1}b\omega$ with $\omega$ a subword such that the exponents on $b$ are $\pm1$. Note that if $\omega$ starts with an $a$, up to homotopy, one finds the first situation in Figure \ref{fig:comu}, meaning that the surface is divided into two regions with the startpoint of the curve in one of them and the continuation in the other one. Thus, the curve will intersect at some point the arcs given by the commutator $ab^{-1}a^{-1}b$, and so $i(\omega)\leq i(ab^{-1}a^{-1}b\omega)-1$. Therefore, $\omega$ has to represent a simple word starting with $a$. Moreover, $\omega$ has to be of the form $a^{n_1}b\cdots a^{n_k}b$, as if it was of the form $a^{n_1}b^{-1}\cdots a^{n_k}b^{-1}$ the curve $ab^{-1}a^{-1}b\omega$ would contain the subword $ba^{n_1}b^{-1}a^{n_2}$ and the only such a word with self-intersection one is $bab^{-1}a$ as proved above, or if $\omega=ab^{-1}a^{-1}b$, the word would not be primitive. 

    \begin{figure}[ht]
\centering
\begin{subfigure}[b]{0.32\textwidth}
\centering
\begin{overpic}[width=.8\textwidth]{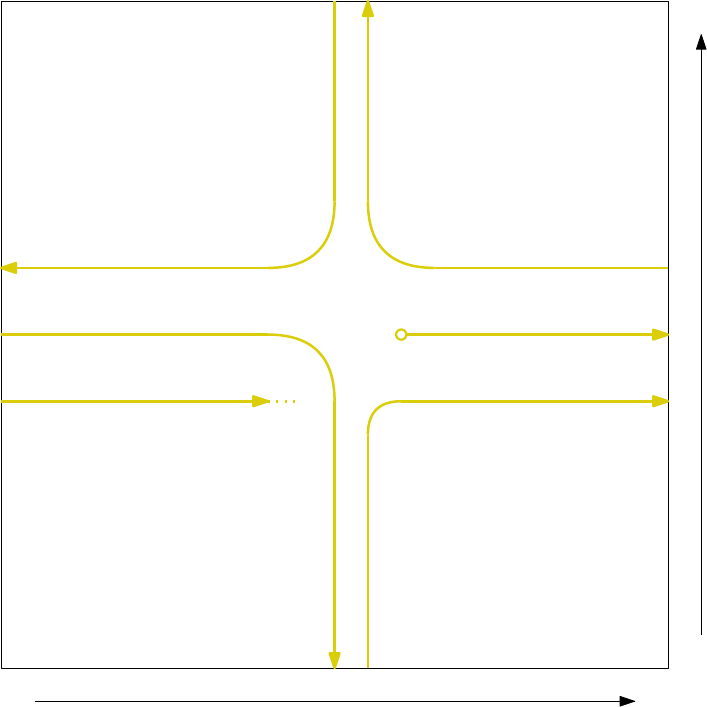}
\put(47,-7){$a$}
\put(102,50){$b$}
\end{overpic}
\label{fig:comu1}
\end{subfigure}
\begin{subfigure}[b]{0.32\textwidth}
\centering
\begin{overpic}[width=.8\textwidth]{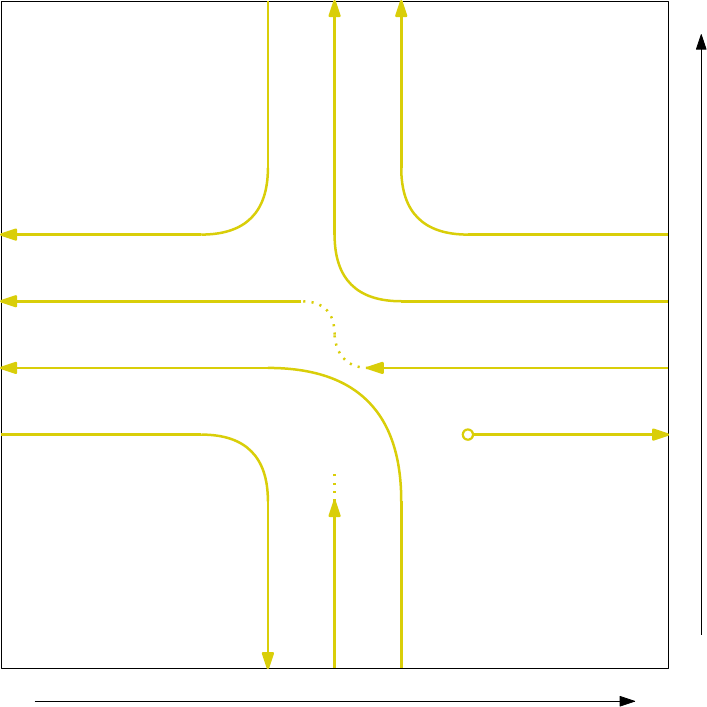}
\put(47,-7){$a$}
\put(102,50){$b$}
\end{overpic}
\label{fig:comu2}
\end{subfigure}
\begin{subfigure}[b]{0.32\textwidth}
\centering
\begin{overpic}[width=.8\textwidth]{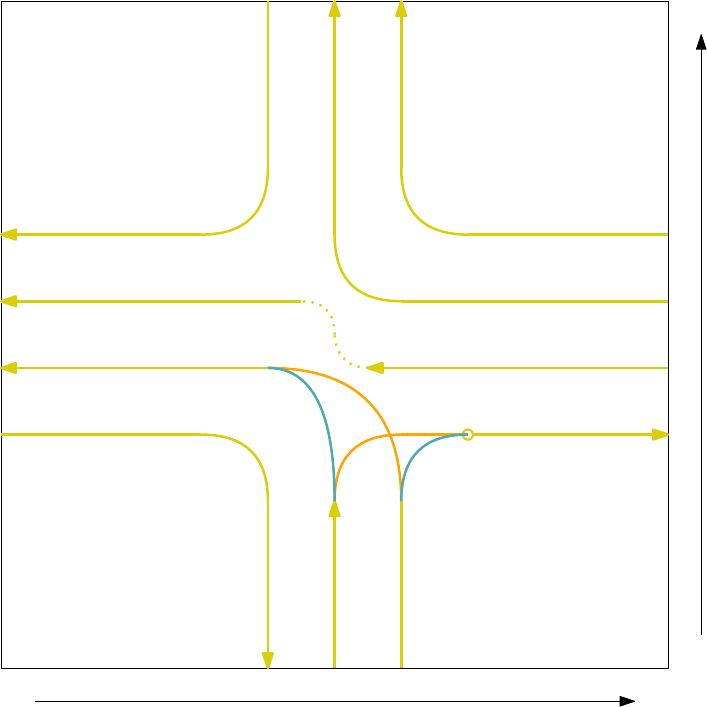}
\put(47,-7){$a$}
\put(102,50){$b$}
\end{overpic}
\label{fig:comu3}
\end{subfigure}
\caption{Possible words starting by $ab^{-1}a^{-1}b$}
\label{fig:comu}
\end{figure}
    
    On the other hand, if $\omega=a^{-i}b^{\epsilon}\omega'$ with $i\geq1$ and $\epsilon = \pm 1$, note first again that for $\epsilon=-1$ it is proved above that the only case containing a subword $a^{-1}ba^{-i}b^{-1}$ with a single self-intersection is $a^{-1}ba^{-1}b^{-1}$ and this is not the case. Therefore, $\epsilon=1$, leading to the second case in Figure \ref{fig:comu}. In this case, all $\omega$ has to be of the form $a^{-n_1}b\cdots a^{-n_{k}}b$, as any change of sign would lead again to the non-possible subword. Therefore, at the last step, one can perform a surgery as in the third part of Figure \ref{fig:comu}, transforming at one component $ab^{-1}ab\omega$ into $\omega$ and lowering the self-intersection by at least $1$, hence $\omega$ must represent a simple curve.

\begin{lemma}\label{lem:uniq}
    Let $ab^{-1}a^{-1}b\omega$ be a word with self-intersection one, being $\omega$ a simple word of the form $a^{n_1}b\cdots a^{n_k}b$. Then, it will always write as  
    \[
    ab^{-1}a^{-1}ba^{m+1}b(a^mb)^{t_1}\cdots a^{m+1}b(a^mb)^{t_s}
    \]
    with $t_1< t_s$ if $m>0$, and as 
    \[
    ab^{-1}a^{-1}b (a^{m+1}b)^{t_1}a^{m}b\cdots (a^{m+1}b)^{t_s}a^{m}b
    \]
    with $t_1>t_s$ if $m<-1$.
\end{lemma}
\begin{proof}
    Start by writing the word as $ab^{-1}a^{-1}b\omega$ with  $\omega=a^{n_1}b\hdots a^{n_k}b$ with $n_i$ being positive for all $i$. The word $ab^{-1}a^{-1}b\omega$ instantly gives a linking pair given the cyclic shifts $\omega_1=bab^{-1}a^{-1}ba^{n_1}b\cdots a^{n_k}$ and $\omega_2=ba^{n_1}b\cdots a^{n_k}bab^{-1}a^{-1}$, satisfying $\omega_2^{-1}<\omega_1<\omega_2<\omega_1^{-1}$.

    Assume now that $n_i\in\{m,m+1\}$ for some $m\geq 1$, as it represents a simple word, and $k>1$, as otherwise the result is trivial.

    Note first that if the curve has self-intersection one, then $n_1=m+1$ and $n_k=m$: if $n_1=m$, then exists $j\in\{1,\hdots,k\}$ such that $n_j=m+1$ giving another linking pair associated to the permutations $\omega_1=ba^{n_j}b\cdots a^{n_k}bab^{-1}a^{-1}ba^{n_1}b\cdots a^{n_{j-1}}$ and $\omega_2=ba^{n_1}b\cdots a^{n_k}bab^{-1}a^{-1}$, rising the total self-intersection to $2$. Similarly, if $n_k=m+1$, there is some $j\in\{1,\hdots,k\}$ such that $n_j=m$ and so we find a linking pair with the permutations $\omega_1=a^{n_k-1}bab^{-1}a^{-1}ba^{n_1}b\cdots ba$ and $\omega_2=a^{n_j}b\cdots a^{n_k}bab^{-1}a^{-1}ba^{n_1}b\cdots a^{n_{j-1}}b$.

    Write now the word as $ab^{-1}a^{-1}ba^{m+1}b(a^mb)^{t_1}\cdots a^{m+1}b(a^mb)^{t_s}$.

    Finally, $t_1<t_s$, as otherwise a new linking pair arises with the permutations $\omega_1=b(a^mb)^{t_s}ab^{-1}a^{-1}ba^{m+1}b(a^mb)^{t_1}\cdots a^{m+1}$ and $\omega_2=b(a^mb)^{t_1}\cdots a^{m+1}b(a^mb)^{t_s}ab^{-1}a^{-1}ba^{m+1}$. 

    The negative case can be proved by noting that after renaming generators and shifting cyclically, $ab^{-1}a^{-1}ba^{n_1}b\cdots a^{n_k}b$ can be rewritten as $ab^{-1}a^{-1}bab^{n_1}\cdots ab^{n_k}$ and there is an analog proof for the case with the exponents on $b$.
\end{proof} 

Note now that the maps $\alpha_m,\tilde{\alpha}_m$ defined in \eqref{eq:auto} fix the set of conjugacy classes $\{[ab^{-1}a^{-1}b],[(ab^{-1}a^{-1}b)^{-1}]\}$. Hence, considering a reduced circular word of the kind $ab^{-1}a^{-1}b\omega$, such that it has a single self-intersection and such that $\omega=a^{n_1}b\cdots a^{n_k}b$ with $n_i\in\{m,m+1\}$ for some $m\neq0,1$ represents a primitive simple word, we can apply the reduction $\alpha_m(ab^{-1}a^{-1}b\omega)=\omega'\alpha_m(\omega)$ if $|\{i\mid n_i=m\}|>|\{i\mid n_i=m+1\}|$ (or with $\tilde{\alpha}_m$ otherwise), where $\omega'$ is a cyclically reduced representative of the conjugacy classes of the commutators of $a,b$. As we are assuming that $\omega$ represents a simple word, applying finitely many times maps of the family $\alpha_m$, $\tilde{\alpha}_m$, this reduction acts as $A$ on the necklace $[n_1,\hdots, n_k]$ defined in Figure \ref{mingkun}, and it will reduce to a word of the kind $A^n(\omega)=a^{n'}b$ up to renaming the generators. Therefore, by inversing $A$ with the appropriate $B_m$'s, and given Lemma \ref{lem:uniq}, after renaming properly the generators such that the representative of the commutator is of the form $ab^{-1}a^{-1}b$, we find that there is a unique permutation of every simple word $\omega$ of the form $a^{n_1}b\cdots a^{n_k}b$ or $a^{-n_1}b\cdots a^{-n_k}b$ making $ab^{-1}a^{-1}b\omega$ a self-intersection one curve. 

Moreover, applying this reduction to Lemma \ref{lem:uniq}, one sees that, in order to obtain a self-intersection one curve from a simple one of the form $a^{n_1}b\cdots a^{n_k}b$ with $n_i\in\{m,m+1\}$ for some $m\neq 0,-1$, the commutator $ab^{-1}a^{-1}b$, has to be inserted before a block $a^{n_i}b$ such that $n_i=m+1$, $n_{i-1}=m$ and maximizing as much as possible the beginning of the chain $\{\sum_{j=1}^s |n_{i+\epsilon j}|\}_{s\geq1}$ with $\epsilon=1$ for positive $m$ and $-1$ for negative (e.g. see Figure~\ref{fig:whercomu}).
\begin{figure}[ht]
\centering
\begin{subfigure}[b]{0.45\textwidth}
\centering
\includegraphics[width=.75\textwidth]{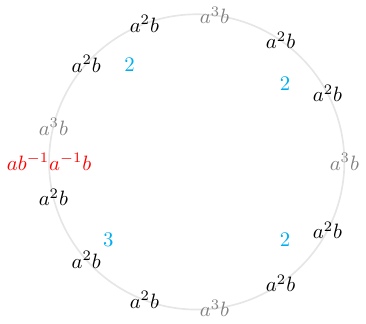}
\end{subfigure}
\begin{subfigure}[b]{0.45\textwidth}
\centering
\includegraphics[width=.75\textwidth]{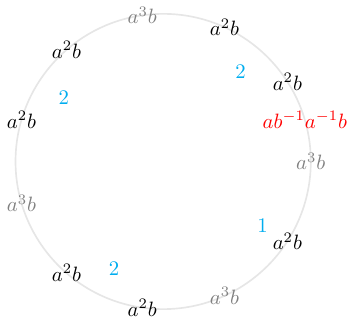}
\end{subfigure}
\caption{Words representing a curve with self-intersection one}
\label{fig:whercomu}
\end{figure}

Finally, up to generator renaming and circular shifting $ab^{-1}a^{-1}ba^{n_1}b\cdots a^{n_k}b$ generates the same curves as $ab^{-1}a^{-1}ab^{-n_1}\cdots ab^{-n_k}$, whilst  $ab^{-1}a^{-1}ab^{n_1}\cdots ab^{n_k}$ generates the same curves as $ab^{-1}a^{-1}ba^{-n_1}b\cdots a^{-n_k}b$, leaving us with the two cases in the statement of this proposition. This concludes the proof.
\end{proof}

Note also that if all the exponents of the letter $b$ are of the same sign, it has been proved already that the only possible word with self-intersection one and with a change of signs in the exponents of $a$ is, up to renaming of the generators, $a^{-1}bab$, which was already considered as an exceptional short case. Thus, now we can assume that there is no change of signs in the exponents and that the exponents of all $b$'s are $1$. Hence, up to renaming, we are left with words of the form 
\[
a^{n_1}ba^{n_2}b\cdots a^{n_k}b
\]
with all exponents being positive.
For the next proposition, we will need an extra definition, which is analogous to Definition~\ref{def:smv} for a single self-intersection.

\begin{definition}[2-variation] \label{def:2v}
Let $m \in \ZZ_{\geq 1}$ and $w = [n_i]_i$ be a necklace with $n_i \in \{m, m+1\}$ for all $i$.
We say that a pair of blocks of $w$ of the same size is an \emph{essential pair} if one block is $(m,  x_2, x_3, \dots, x_{k-1}, m)$ and the other block is $(m+1,  x_2, x_3, \dots, x_{k-1}, m+1)$ where $x_i \in \{ m, m+1\}$.
We say that $w$ has \emph{$2$-variation} if among all pairs of blocks of $w$, only one is essential.
\end{definition}

\begin{remark}\label{rmk:svform}
    Such a name is given after Buser--Semmler's small variation since we consider this to be the smallest case without small variation, hence variation $2$: an essential pair of blocks will always look like
    \[
    B_1= (m, x_2, x_3, \dots, x_{k-1}, m),
    \qquad
    B_2 = (m+1, x_2, x_3, \dots, x_{k-1}, m+1),
    \]
    with $x_i \in\{ m, m+1\}$.
    Therefore, 
    \[
    \bigg| \sum_{S_1} n_i - \sum_{S_2} n_j \bigg|
    =
    2,
    \]
    and whenever the small variation condition is broken, there will always exist at least one essential pair of blocks.
\end{remark}

\begin{proposition} \label{prop:si1case3}
The only circular words reprsenting a curve of a single self-intersection of the form $a^{n_1} b \cdots a^{n_k} b$ with all exponents being positive are
\begin{itemize}
    \item
    $a^{m}b \, a^{m+2}b$ where $m \in \ZZ_{\geq 1}$,

    \item
    $a^{n_1} b \cdots a^{n_k} b$ where the exponent necklace $[n_1, \dots, n_k]$ has $2$-variation.
\end{itemize}
\end{proposition}

\begin{proof}
    Let us start by assuming that there are at least two exponents of $a$ with a difference of $3$, i.e.\ there are $i,j,m\geq1$ such that $n_i=m$, $n_j\geq m+3$. Take then the cyclic shifts of the word $\omega_1=a^{n_{i}}b\cdots$, $\omega_2=a^{n_j-1}ba^{n_{j+1}}\cdots ba$, and $\omega_3=a^{n_j-2}ba^{n_{j+1}}\cdots ba^2$. By the lexicographic ordering, one has 
    \[
        \omega_k<\omega_1<\omega_k^{-1}<\omega_1^{-1}\quad\text{for}\quad k\in\{2,3\},
    \]
    which in \cite{Cohen-Lustig} gives $2$ linking pairs associated to the pairs $(\omega_2,\omega_1)$, and $(\omega_3,\omega_1)$, in different classes, and hence self-intersection of at least $2$. 

    Let us now move to the case where there are $i,j,m\geq1$ such that $n_i=m$, $n_j= m+2$. Note that, from the existence of such a pair, we can straightforwardly take two cyclic shifts of the word: $\omega_1=a^{n_{i}}b\cdots a^{n_{i-1}}b$ and $\omega_2=a^{n_j-1}ba^{n_{j+1}}b\cdots ba$, given by the same ordering as above $\omega_2<\omega_1<\omega_2^{-1}<\omega_1^{-1}$, hence one linking pair giving one self-intersection. Now, if there was a $1\leq k\neq i,j$ such that $n_k\neq m+1$, by the same procedure we would find an extra self-intersection, giving at least two. Thus, assume now the word is of the form 
    \[
    a^{n_1}b\cdots a^{n_k}b a^{m} ba^{m_1}b\cdots a^{m_{k'}}ba^{m+2}b,
    \]
    with $n_i,m_j=m+1$. If $k'\neq 0$, choosing the cyclic shifts $\omega_1=a^mba^{m_1}\cdots a^{n_k}b$ and $\omega_2=a^{m_{k'}-1}ba^{m+2}\cdots ba$, again one has $\omega_2<\omega_1<\omega_2^{-1}<\omega_1^{-1}$ giving another linking pair in a different class, and hence an extra self-intersection, i.e. at least $2$. If $k'=0$ and $k\neq0$, one can find analogously an extra intersection. Implying then that if there are two exponents with a difference of two, the only candidate with a single self-intersection (up to renaming and cyclic shift) is when $k=k'=0$, that is $ab^{m}ab^{m+2}$, and by \cite[Proposition~A.1]{Chas-Phillips}, indeed has a single self-intersection for every $m\in\ZZ_{\geq1}$.

    Let us now move to the case where we have a word of the form $a^{n_1}b\cdots a^{n_k}b$ with all $n_i\in\{m,m+1\}$ for some $m\in\ZZ_{\geq1}$. We want to prove that it has self-intersection one if and only if the necklace of positive integers $[n_1,\hdots, n_k]$ has $2$-variation.

    It is enough to see that in this case there is a $1:1$ correspondence between classes of linking pairs in \cite{Cohen-Lustig} algorithm and essential pairs of blocks in the necklace $[n_1,\hdots, n_k]$. We will start by constructing a class of linking pairs in the word given an essential pair of blocks in the associated necklace.

    Recall from Remark \ref{rmk:svform} such a pair of blocks is always going to be of the form 
    \[
    S_1=\{m,n_{i_1+2},\hdots,n_{i_1+s-1},m\},
    \qquad
    S_2=\{m+1,n_{i_2+2},\hdots,n_{i_2+s-1},m+1\},
    \]
with $n_{i_1+j}=n_{i_2+j} \in \{ m, m+1 \}$ for $j=2,\hdots, s-1$. Note that every such pair gives rise to a linking pair on the word $a^{n_1}b\cdots a^{n_k}b$ by taking $\omega_1=a^m ba^{n_{i_1+2}}b\cdots b$, $\omega_2=a^m ba^{n_{i_2+2}}b\cdots ba$ with $\omega_2<\omega_1<\omega_2^{-1}<\omega_1^{-1}$ with respect to the lexicographic order $a<b<a^{-1}<b^{-1}$. Note also that by this construction every pair of blocks gives rise to a different class of linking pairs, given by the difference in first and last number, and that two linking pairs $(i,j)\sim (i+1,j+1)$ are in the same class if and only if the $i$th and $j$th letters of the word representative start by the same letter, being the only equivalence between linking pairs possible in this case. Hence, we found an injection from the set of essential pairs of blocks into self-intersections of the curve represented by the word. 

The converse injection is given by the following: take a linking pair $(i_1,i_2)$, that is we have two permutations of the initial word giving 
\begin{equation}\label{si:cd}
    \omega_{i_1}<\omega_{i_2}<\omega_{i_1}^{-1}<\omega_{i_2}^{-1}.
\end{equation}
Assume that both start with $a$, then they are of the form 
$\omega_{i_t}=a^{l_t}ba^{n_{j_t+1}}b\cdots$ and $\omega_{i_t}^{-1}=a^{-(n_{j_t}-l_t)}b^{-1}a^{-n_{j_t-1}}b^{-1}\cdots$ for $t=1,2$.
Now, first inequality of Equation \eqref{si:cd} implies that either $l_1\geq l_2+1$, or $l_1=l_2$ and $n_{j_1+r}=n_{j_2+r}$ for $r=1,\hdots,s-1$ for some $s$ and $n_{j_1+s}=n_{j_2+s}+1$. Similarly, the last inequality of the equation gives for the other side $n_{j_1}-l_1\geq n_{j_2}-l_2+1$ or $n_{j_1}-l_1= n_{j_2}-l_2$ and $n_{j_1-r}=n_{j_2-r}$ for $r=1,\hdots,s'-1$ for some $s'$ and $n_{j_1-s'}=n_{j_2-s'}+1$. Note that since all $n_i\in\{m,m+1\}$, at most $l_1=l_2+1$ and so in all cases one finds at the extremes two subwords of the form $ba^{m+1}ba^{m_1}b\cdots ba^{m_{s}}ba^{m+1}b$, $ba^{m}ba^{m_1}b\cdots ba^{m_{s}}ba^{m}b$, that give an essential pair of blocks on the necklace of positive integers. 

Moreover, when both words start with a $b$, then the same argument for $l_1=l_2=0$ applies. Finally, assume they start with different letters. By Equation \eqref{si:cd}, $\omega_{i_1}$ starts with $a$ and $\omega_{i_2}$ with $b$, i.e. $\omega_{i_1}=a^lba^{n_{j_1+1}}b\cdots$ for some $0<l<n_{j_1}$ and $\omega_{i_2}=ba^{n_{j_2}+1}b\cdots$. In this case, $\omega_{i_1}^{-1}=a^{-(n_{j_1}-l)}b^{-1}a^{n_{j_1-1}}b\cdots$ and $\omega_{i_2}^{-1}=a^{-n_{j_2}}b^{-1}a^{n_{j_1-1}}b\cdots$.The first two inequalities of Equation \eqref{si:cd} are automatically true. The third inequality implies $n_{j_1}-l\geq n_{j_2}$, and since they can only differ by one it can only happen if $l=1$, $n_{j_1}=m+1$, and $n_{j_2}=m$. Then, all $n_{j_1-r}=n_{j_2-r}$ for $r=1,\hdots, s-1$ until some $n_{j_1-s}=m+1$ and $n_{j_2-s}=m$, for which we get again an essential pair of blocks in the necklace of positive integers. Moreover, note that by the construction above all the linking pairs giving the same essential pair of blocks are in the same class, as again the only possible equivalence is $(i,j)\sim (i+1,j+1)$ are in the same class if and only if the $i$th and $j$th letters of the word representative start by the same letter. 
\end{proof}

\begin{remark}
    The proofs of the above propositions have inside all the steps to prove that Theorem \ref{thm:bus} follows from \cite{Cohen-Lustig} algorithm. 
\end{remark}

\section{Counting}
This section is dedicated to counting primitive curves with self-intersection one, i.e. the following proof.

\begin{proof}[Proof of Theorem~\ref{theorem:si1}]
The exceptional cases are for length $4$ and $5$. For length $4$ there are two cases: $a^2b^2$ and $aba^{-1}b$, which after renaming of generators give us $8$ conjugacy classes, i.e.\ $|\{\gamma \in\mathcal{PC}_1(\Sigma_{1,1}\mid i(\gamma) = 1,\, \ell_\omega(\gamma)=4\}|=8$, and for length $5$ we have only the words of type $ab^{-1}a^{-1}b^2$ which give $|\{\gamma \in \mathcal{PC}_1(\Sigma_{1,1})\mid i(\gamma) = 1, \, \ell_\omega(\gamma)=5\}|=8$. 

For the general case of length $L\geq 6$, name $P_1(L)$ the number of words arising from Proposition \ref{prop:si1case1}, $P_2(L)$ the number of words arising from Proposition \ref{prop:si1case3} of the form $ab^{m}ab^{m+2}$, and $P_3(L)$ the number of the rest words arising from Proposition \ref{prop:si1case3}. Then, for $L\geq 6$, we will have 
\[
|\{\gamma \in\mathcal{PC}_1(\Sigma_{1,1})\mid i(\gamma) = 1, \, \ell_\omega(\gamma)=L\}|=8\cdot (P_1(L)+P_2(L)+P_3(L)),
\]
since each of these words will give different conjugacy classes after all the possible renamings of generators $\{a,b,a^{-1},b^{-1}\}$.

It is straightforward from the proposition that $P_1(L)=2\cdot|\{$aperiodic necklaces of positive integers $[n_1,\hdots,n_k]$ with small variation such that $k+\sum_{i=1}^kn_i=L-4\}|$. Hence, from the proof of Theorem \ref{mainth1}, that is 
\[
    P_1(L)=2\cdot\sum_{d \mid (L-4)}\mu(d)\left\lfloor \frac{L-4}{2d}\right\rfloor=\varphi(L-4),
\]
where again $\mu$ and $\varphi$ denote the Möbius function and Euler's totient function, respectively.

Now, $P_2(L)=1$ for even $L$ and vanishes otherwise. Lastly, to count $P_3(L)$, we need to count the number of $2$-variation necklaces of positive integers that give a word of length $L$.

\begin{proposition}
    Let $m, x, y \in \ZZ_{\geq 1}$.
    If $\gcd(x, y) = 2$, then there exists a unique necklace of integers with $2$-variation that contains exactly $x$ occurrences of the number $m$, and $y$ occurrences of the number $m+1$.
    Otherwise, no such necklaces exist.
\end{proposition}
\begin{proof}
Without loss of generality, we assume throughout the proof that $x\leq y$.
Let $k \in \ZZ_{\geq 1}$, and let $[n_i]_i = [n_1, \dots, n_k]$ be a necklace such that $n_i \in \{ m, m+1 \}$ for all $1 \leq i \leq k$, $|\{ i \mid n_i = m \}| = x$, and $|\{ i \mid n_i = m + 1 \}| = y$.

First, we prove that if $x \mid y$ then the necklace $[n_i]_i$ does not have variation $2$.
If every run of $m+1$ has size $y/x$, then $[n_i]_i$ has small variation.
If there exist two runs of $m+1$ with sizes differing by at least $3$ (for example, $y/x - 1$ and $y/x + 2$),
then $[n_i]_i$ does not have variation $2$.
Thus, if $[n_i]_i$ has variation $2$, then the sizes of its runs of $m+1$ can only take values in $\{ y/x-1, y/x, y/x+1 \}$, as the sum of the runs have to sum $y$ and there are $x$ of them.

Further, if there is one run of size $y/x+1$, then there is at least one run of size $y/x-1$ and there can only be one as otherwise, these give immediately two essential pairs of blocks. Therefore, for a necklace with $2$-variation, there is one run of size $y/x-1$, one of size $y/x+1$, and $x-2$ runs of size $y/x$. 

If $x>2$, then the necklace cannot have $2$-variation, as one essential pair of blocks is given straight by the runs with difference $2$, i.e.\ $\{m,m+1,\overset{y/x-1}{\dots},m+1,m\}$ and $\{m+1,\overset{y/x+1}{\dots},m+1\}$, and another one by extending these sets to an adjacent gap of size $y/x$ (e.g.\ see Figure \ref{fig:necxgeq2}). When $x=2$, there are only two runs, of sizes $y/x-1$ and $y/x+1$, and one unique possible configuration with these, giving always variation $2$ (see Figure~\ref{fig:necx2}).

\begin{figure}[ht]
\centering
\begin{subfigure}[b]{0.45\textwidth}
\centering
\includegraphics[width=.75\textwidth]{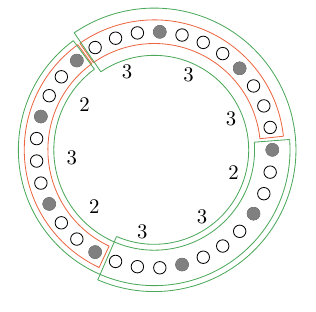}
\caption{$x=3, y=6$,}
\label{fig:necxgeq2}
\end{subfigure}
\begin{subfigure}[b]{0.45\textwidth}
\centering
\includegraphics[width=.7\textwidth]{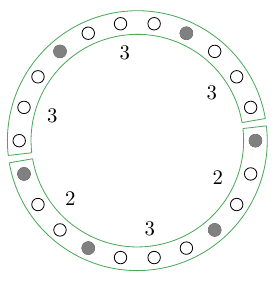}
\caption{$x=2, y=4$.}
\label{fig:necx2}
\label{fig:xdivy}
\end{subfigure}
\caption{Two examples of cases where $x \mid y$}
\end{figure}

Assume from now on $x \nmid y$. First note that for a necklace with variation $2$ the sizes of its runs have to take values in $\{\floor{y/x},\ceil{y/x}\}$: the existence of a run of size at least $\ceil{y/x}+1$ implies that there should be at most one run of size $\floor{y/x}$, as every couple of runs with difference two gives rise to an essential pair of blocks, and the rest be of size $\ceil{y/x}$. However, there are $x$ runs and so 
$y=(\ceil{y/x}+1)+\floor{y/x}+(x-2)\ceil{y/x}=x\ceil{y/x}$, contradicting $x\nmid y$. The case for the existence of a run of size at most $\floor{y/x}-1$ is symmetric.

Moreover, the necklace $[n_i]_i$ has $2$-variation if and only if the associated run necklace $A[n_i]_i$ described in Figure~\ref{mingkun} has. This comes naturally from the map $A$ as two blocks of the form $ S_1=\{m, n_{j+2},\hdots, n_{j+s-1},m\}$ and $S_2=\{m+1, n_{j+2},\hdots, n_{j+s-1},m+1\}$ will map to the same sequence of runs with the first and last one being bigger by one at the second case. Conversely, if there is such a sequence of runs, finding the associated pair of blocks in $[n_i]_i$ is straightforward. 

Finally, note that by the same computations in Lemma \ref{lem:min}, a necklace of profile $(m,x,y)$ maps to a necklace of profile $(\floor{y/x},x-y+x\floor{y/x},y-x\floor{y/x})$, and by elementary properties $\gcd(x,y)=\gcd(x-y+x\floor{y/x},y-x\floor{y/x})$. Therefore, as the number of appearances in the profile keeps decreasing while they do not divide each other, this will only stop when they do, and as the $\gcd$ is maintained, that will happen when the minimum reaches $\gcd(x,y)$, and as the dividing case has already been proved, there will exist a $2$-variation necklace if and only if $\gcd(x,y)=2$.
\end{proof}

Hence, since a word $ab^{n_1}\cdots ab^{n_k}$ has length $k+\sum_{i=1}^kn_i$, $P_3(L)$ will be exactly the number of solutions to the following equation. 

\begin{proposition} \label{prop:='}
    Consider the equation
    \begin{equation} \label{eq:K'}
        x(m+1) + y(m+2) = L
    \end{equation}
    where $L \in \ZZ_{\geq 1}$ is given, and $x,y,m \in \ZZ_{\geq 1}$ are unknown such that $\gcd(x,y) = 2$.
    Then for even $L$, there are exactly $\ceil{\varphi(L/2)/2}-1$ solutions,
    and none for odd $L$. 
\end{proposition}

\begin{proof}
Let us start by introducing some notation.
Define
\begin{align*}
    S(L)
    & \coloneqq
    \{ (x, y, m) \in \ZZ_{\geq 1}^3 \mid x(m+1) + y(m+2) = L \}, \\
    S_{>1}(L)
    & \coloneqq
    \{ (x, y, m) \in S(L) \mid \gcd(x, y) > 1 \}, \\
    S_{1}(L)
    & \coloneqq
    \{ (x, y, m) \in S(L) \mid \gcd(x, y) = 1 \}, \\
    S_{2}(L)
    & \coloneqq
    \{ (x, y, m) \in S(L) \mid \gcd(x, y) = 2 \}.
\end{align*}
Our objective is to determine $|S_2(L)|$.
The set of solutions $S_2(L)$ is empty if $L$ is odd.
When $L$ is even, the mapping $(x,y) \mapsto (x/2, y/2)$ defines a bijection from $S_2(L)$ to $S_1(L)$, and hence, we have $|S_2(L)| = |S_1(L/2)|$ for $L$ even.
By Proposition~\ref{prop:=}, we have
\[
    |S_{>1}(L)| + |S_1(L)|
    =
    \floor{L/2} - \sigma_0(L) + 1,
\]
where $\sigma_0(L)$ denotes the number of divisors of $L$.
(Here, we exclude solutions where $y=0$.)
Since any solution $(x, y, m) \in S_{>1}(L)$ corresponds to a solution in $S_1(L/\gcd(x,y))$,
we have
\[
    |S_{>1}|
    =
    \sum_{d \mid L, d \neq 1} |S_1(L/d)|.
\]
Therefore, we have
\begin{equation} \label{eq:rec}
    \sum_{d \mid L} |S_1(L/d)|
    =
    \floor{L/2} - \sigma_0(L) + 1,
\end{equation}
and in particular, we have
\begin{equation} \label{eq:p}
    |S_1(p)|
    =
    \floor{p/2},
    \qquad
    \text{for any prime $p$.}
\end{equation}
Now, note that since $S_1(0) = 1$, $|S_1(L)|$ can be uniquely determined by \eqref{eq:rec} and \eqref{eq:p} for any $L \in \ZZ_{\geq 1}$.
On the other hand, a direct computation shows that the function defined by $L \mapsto \ceil{\varphi(L)/2} - 1$ maps $0$ to $1$, and satisfies \eqref{eq:rec} and \eqref{eq:p}.
Hence, for any $L \in \ZZ_{\geq 1}$, we have
\[
    |S_1(L)|
    =
    \ceil{\varphi(L)/2} - 1
\]
and therefore, for any $L$ even, we have
\[
    |S_2(L)|
    =
    |S_1(L/2)|
    =
    \ceil{\varphi(L/2)/2} - 1.
\]
The proposition follows.
\end{proof}

Concluding, Proposition \ref{prop:='} implies that $P_3(L)=\ceil{\varphi(L/2)/2}-1$ for even $L$, and zero otherwise, hence one gets for $L\geq4$,
\begin{align*}
|\{\gamma \in\pi_1(S)\mid \ell_\omega(\gamma)=L, i(\gamma)=1\}|&=8\cdot(P_1(L)+P_2(L)+P_3(L))\\
&=
\begin{cases}
8 \, \varphi(L-4) & \text{for odd }L, \\
8 \left(\varphi(L-4) + 1 + \left\lceil \varphi(L/2)/2 \right\rceil - 1 \right)& \text{for even }L.
\end{cases}
\end{align*}
by summing the general cases and checking that it coincides for $L=4,5$ with the convention $\varphi(0)=0$.
\end{proof}

\part{All curves} \label{section:all}

In this section, we shall count all closed curves on $\Sigma_{1,1}$ of given word length, regardless of their self-intersection numbers.

This result (Theorem~\ref{mainth3}) is expected to be elementary.
However, despite our search in the literature, we couldn't find a reference.
Hence, we provide a complete proof here.

\begin{proof}[Proof of Theorem~\ref{mainth3}]
Let $n \in \ZZ_{\geq 1}$.
Let us denote by $w_n$ the number of reduced words in $\{ a, b, a^{-1}, b^{-1} \}$ of length $n$.
Define generating functions
\[
    W(t)
    \coloneqq
    \sum_{n=1}^{\infty}
    w_n t^n.
\]
Every word $\omega$ under consideration can be written in the form
\[
    x_1^{n_1} x_2^{n_2} \cdots x_k^{n_k}
\]
where for all $i$, $x_i \in \{ a, a^{-1}, b, b^{-1} \}$, $n_i \in \ZZ_{\geq 1}$, and $x_{i+1} \notin \{ x_{i}, x_{i}^{-1} \}$.
We discuss based on the parity of $k$.
If $k$ is even, then the last letter $x_k$ must be $b$ or $b^{-1}$.
Thus
\[
    \sum_{n=1}^{\infty}
    w_{2n} t^{2n}
    =
    4 \sum_{i=1}^{\infty} \frac{t}{1-t} \left( \frac{2t}{1-t} \right)^{2i-1}
    =
    \frac{8 t^2}{-3t^2-2t+1}.
\]
This arises from the following reasoning.
We assume $x_1 = a$.
The exponent $n_1$ can be any positive integer, giving a factor $t + t^2 + \cdots = t/(1-t)$.
Next, $n_2$ can also be any positive integer, and $x_2$ can be chosen between $b$ and $b^{-1}$.
This gives a factor $2t/(1-t)$, and so on so forth.
Finally, $x_1$ can also be $a^{-1}$, $b$, or $b^{-1}$, which gives a factor of $4$.

If $k$ is odd and $k \neq 1$, then we have $x_k = x_1$.
So by a similar argument, we have
\[
    \sum_{n=1}^{\infty}
    w_{2n+1} t^{2n+1}
    =
    4\sum_{i=1}^{\infty} \frac{t}{1-t} \left( \frac{2t}{1-t} \right)^{2i-1} \frac{t}{1-t}
    =
    \frac{8t^3}{(1-t)(-3t^2 - 2t + 1)}.
\]
Therefore,
\[
    W(t)
    =
    \frac{8t^2}{-3t^2-2t+1} + 
    \frac{8t^3}{(1-t)(-3t^2 - 2t + 1)} +
    \frac{4t}{1-t}
    =
    \frac{4t-12t^3}{(1-t)(-3t^2 - 2t + 1)}
\]
where the term $1/(1-t)$ corresponds to the case $k = 1$.
This can be rewritten as
\[
    W(t)
    =
    -4 + \frac{2}{1 - t} + \frac{1}{1 + t} + \frac{1}{1 - 3t}
    =
    \sum_{n=1}^{\infty} (2 + (-1)^n + 3^n) \, t^n,
\]
and hence, for any $n \in \ZZ_{\geq 1}$, we have 
\[
    w_n
    =
    2 + (-1)^n + 3^n.
\]

Now, the Möbius inversion formula implies that the number of primitive reduced words of length $n$ is equal to
\[
    \sum_{d \mid n}
    \mu(d) \big( 2+(-1)^{n/d}+3^{n/d} \big),
\]
and therefore, as primitive words of length $n$ and primitive words of length $n$ are in $n$-to-$1$ correspondence, the number of primitive reduced necklaces in $\{ a,b,a^{-1},b^{-1} \}$ is
\begin{equation} \label{eq:pNec}
    \frac{1}{n}
    \sum_{d \mid n}
    \mu(d) \big( 2+(-1)^{n/d}+3^{n/d} \big)
    =
    \begin{cases}
        4 & \text{if } n = 1,2, \\
        \displaystyle \frac{1}{n} \sum_{d \mid n} \mu(d) \, 3^{n/d} & \text{if } n \geq 3
    \end{cases}
\end{equation}
where we have used the arithmetic identities
\[
    \sum_{d \mid n} \mu(d)
    =
    \begin{cases}
        1 & \text{if } n = 1, \\
        0 & \text{if } n \geq 2,
    \end{cases}
    \quad
    \text{and}
    \quad
    \sum_{d \mid n} \mu(d) (-1)^{n/d}
    =
    \begin{cases}
        -1 & \text{if } n = 1, \\
        2 & \text{if } n = 2, \\
        0 & \text{if } n \geq 3. \\
    \end{cases}
\]
Now, \eqref{eq:pNec} implies that (by summing over all factors of $n$) the number of (not necessarily primitive) reduced necklaces of length $n$ is
\[
\frac{3 + (-1)^{n}}{2}
+
\frac{1}{n} \sum_{d \mid n} \varphi(d) \, 3^{n/d}.
\]
This completes the proof.
\end{proof}

\end{document}